\documentclass[onefignum,onetabnum,article]{siamart171218}

\usepackage{amsmath,amssymb, graphicx,latexsym}

\usepackage{array}
\usepackage{caption}
\usepackage{url}
\usepackage{color}
\usepackage{epstopdf}

\usepackage{algorithm}
\usepackage{algorithmic}
\usepackage{mathtools}
\usepackage{multirow}
\usepackage{float}
\usepackage{geometry}
\usepackage{mathptmx}      
\usepackage{latexsym}

\usepackage{amsfonts}
\usepackage{graphicx}
\usepackage{epstopdf}
\usepackage{algorithmic}

\usepackage{tikz}
\newtheorem{Thm}{Theorem}[section]
\newtheorem{Def}{Definition}[section]
\newtheorem{Le}{Lemma}[section]
\newtheorem{Coro}{Corollary}[section]
\newtheorem{Rem}{Remark}[section]

\numberwithin{equation}{section}

\ifpdf
\DeclareGraphicsExtensions{.eps,.pdf,.png,.jpg}
\else
\DeclareGraphicsExtensions{.eps}
\fi

\newcommand{\be}[1]{\begin{equation}\label{#1}}
	\newcommand{\ee}{\end{equation}}

\newsiamremark{remark}{Remark}
\newsiamremark{example}{Example}
\newsiamremark{hypothesis}{Hypothesis}
\crefname{hypothesis}{Hypothesis}{Hypotheses}
\newsiamthm{claim}{Claim}

\headers{Stability for a multi-frequency inverse random source problem}{T. Wang, X. Xu and Y. Zhao}

\title{Stability for a multi-frequency inverse random source problem}

\author{Tianjiao Wang\thanks{School of Mathematical Sciences, Zhejiang University, Hangzhou, Zhejiang, China, 310058. (\email{wangtianjiao@zju.edu.cn}).
	}
	\and Xiang Xu\thanks{School of Mathematical Sciences, Zhejiang University, Hangzhou, Zhejiang, China,310058. XX's work was supported in part by National Key Research and Development Program of China (2023YFA1009100) and National
		Natural Science Foundation of China (12071430), Key Laboratory
		of Collaborative Sensing and Autonomous Unmanned Systems of Zhejiang Province.  (\email{xxu@zju.edu.cn}). }
	\and Yue Zhao \thanks{ School of Mathematics and Statistics, and Key Lab NAA--MOE, Central China Normal University,
		Wuhan 430079, China.
		(\email{zhaoyueccnu@163.com}).
	}
}

\usepackage{amsopn}

\makeatletter
\newcommand*{\addFileDependency}[1]{
	\typeout{(#1)}
	\@addtofilelist{#1}
	\IfFileExists{#1}{}{\typeout{No file #1.}}
}
\makeatother

\begin{document}
	\maketitle

		\begin{abstract}

We present stability estimates for the inverse source problem of the stochastic Helmholtz equation in two and three dimensions by either near-field or far-field data.
The random source is assumed to be a microlocally isotropic generalized Gaussian random function. 
For the direct problem, by exploring the regularity of the Green function, we demonstrate that the direct problem admits a unique bounded solution with an explicit integral representation, which enhances the existing regularity result.
For the case using near-field data, the analysis of the inverse problem employs microlocal analysis to achieve an estimate for the Fourier transform of the micro-correlation strength by the near-field correlation data and a high-frequency tail. The stability follows by showing the analyticity of the data and applying a novel analytic continuation principle.
The stability estimate by far-field data is derived by investigating the correlation of the far-field data. The stability estimate consists of the Lipschitz type data discrepancy and the logarithmic stability. The latter decreases as the upper bound of the frequency increases,
which exhibits the phenomenon of increasing stability.

		\end{abstract}
		\begin{keywords}
		increasing stability, inverse source problem, generalized Gaussian random function.
		\end{keywords}
		
	\begin{AMS}
		35Q74, 35R30, 78A46
	\end{AMS}
	
	\section{Introduction}\label{intro}
Inverse source scattering problem is concerned with recovering the unknown source from near-field or far-field data away from its support. Such problem has generated tremendous interest due to its wide applications in scientific and engineering fields such as seismology, telecommunications, medical imaging, antenna synthesis, radar technology, and magnetoencephalography \cite{arridge1999optical,bao2002inverse,bao2010multi,isakov1990inverse}. However, the non-uniqueness of the inverse source problem at a single frequency, caused by the existence of non-radiating sources, poses a challenge \cite{bleistein1977nonuniqueness,devaney1982nonuniqueness}. Consequently, additional information is required for a unique determination of the source. To resolve this issue, the use of multi-frequency data has been realized to be an effective approach to regain the uniqueness and achieve enhanced  stability \cite{bao2020stability,bao2010multi, 2016increasing, li2020stability}. 

	In many applications, the source is often considered as a random field due to uncertainties in the surrounding environment or random measurement noise \cite{badieirostami2010wiener}. The presence of randomness introduces additional challenges compared to deterministic source scattering. Specifically, the regularities of wave fields tend to be lower, and the measurements become statistical data. 
	Inverse source problems driven by Wiener process have been extensively investigated \cite{bao2016random,bao2017random,bao2014random,li2017random,li2017randomstab}. 
	In recent studies, uniqueness of inverse source problems have been studied in \cite{Helin2020inverse, li2022far, li2021inverse} by assuming the source to be a  generalized Gaussian random field.
	The covariance of such random field is a classical pseudo-differential operator with a principal symbol taking the form $h(x)|\xi|^{-m}$, where $h$ is called micro-correlation strength. This model encompasses various important stochastic processes, including white noise, fractional Brownian motion, and Markov fields \cite{lassas2008inverse}. Compared with the many uniqueness results of the inverse random source problems, the stability has been much less studied.
	To the best of our knowledge, the only existing stability results were obtained in \cite{li2023stability, li2017increasing} \textcolor{black}{driven by Wiener process}. The corresponding stability
	for the generalized Gaussian random field remains unsolved.
	
	In numerical experiments in \cite{bao2016random, li2017random}, it has been observed that the ill-posedness of the inverse random source problem can be overcome by using multi-frequency data which yields increasing stability, i.e.,
	as the frequency increases the inverse problem becomes more stable. The goal of this work is to mathematical verify the increasing stability with a generalized Gaussian random field. Specifically, consider the stochastic Helmholtz equation \begin{equation} \label{eq3.1}
		\Delta u+ k^2 u=f, \quad \mathrm{in} \quad \mathbb{R}^d, \, d = 2, 3,
	\end{equation} where $k>0$ is the wavenumber. The source function $f$ is a microlocally isotropic general Gaussian random field of order $m$ in a bounded \textcolor{black}{Lipschitz} domain $D \subset \mathbb{R}^d$ (see Section \ref{pre} for a detailed definition). The wave field $u$ is required to satisfy the Sommerfeld radiation condition \begin{equation}\label{eq3.2}
		\lim_{r \to \infty} r^{\frac{d-1}{2}} (\partial_r u- ik u)=0,\quad r=|x|.
	\end{equation}
	Let $B_R = \{x: x\in\mathbb R^d, \, |x|\leq R\}$ with boundary $\partial B_R$ and assume that $D\subset\subset B_R$. The inverse problem is to determine  the principle symbol of $f$ from either the near-field data on $\partial B_R$ or far-field data.

	Compared with the deterministic case, new challenges arise due to the roughness and randomness of the source. The well-posedness of the direct problem
	\eqref{eq3.1}--\eqref{eq3.2} has been discussed in \cite{li2022far}. However, it is not clear if the solution can be represented by the convolution of the Green function
	and the source as in the classical setting. By studying the regularity of the Green function, we establish an explicit integral representation of the solution which holds pointwisely. As a consequence, we obtain an enhanced regularity result.
	The analysis of the inverse problem employs microlocal analysis to achieve an estimate for the Fourier transform of the micro-correlation strength by the near-field correlation data and a high-frequency tail. Next, we show that the correlation data is analytic and derive an upper bound with respect to complex wave number.
	The stability estimate follows by an application of a novel analytic continuation developed in \cite{zhai2023increasing}. For the case of far-field data, the stability estimate can be derived by investigating the
	correlation of the far-field data. The stability has a unified form which consists of the Lipschitz data discrepancy and a logarithmic stability, where the latter
	decreases as the frequency increases.

The rest of this paper is organized as follows.  Section \ref{pre} is devoted to the definition and some properties of the microlocally isotropic Gaussian random function. In Section \ref{direct}, we derive a regularity result for the Green function of the Helmholtz equation which leads to an explicit integral representation of the direct problem. The main increasing stability results are presented in Section \ref{near} and Section \ref{far} for near-field data and far-field data, respectively.  A conclusion is given in Section \ref{conclusions}. Throughout this paper, $a \lesssim b $ stands for $a \le C b$, where $C>0$ is a generic constant whose special value is not required but should be clear from the context.
	
	\section{Preliminaries}\label{pre}
	In this section, we state the properties of microlocally isotropic generalized Gaussian random functions. Let $(\Omega, \mathcal{A}, \mathbb{P})$ be a complete probability space. Denote the test function space by $\mathcal{D}$, which consists of smooth functions with compact supports in $\mathbb{R}^d(d=2,3)$. Then the dual space of $\mathcal{D}$ is denoted as $\mathcal{D}'$. A scalar field $f$ is said to be a real-valued generalized Gaussian random function if $f\,:\, \Omega \to \mathcal{D}'$ is a distribution such that 
	for each $\omega\in\Omega$ the path $f[\cdot](\omega)\in \mathcal{D}'$ is a linear functional on $\mathcal{D}$, and
	$\omega \mapsto \langle f(\omega), \phi \rangle$ is a real-valued Gaussian random variable for all $\phi \in \mathcal{D}$. The expectation of $f$ is a generalized function defined by \[
	\mathbb{E}f\, :\, \phi \mapsto \mathbb{E}\langle f,\phi \rangle, \quad \phi \in \mathcal{D},
	\] and the covariance is a bilinear form given by \[
	\mathrm{Cov}f\,:\, (\phi_1, \phi_2) \mapsto \mathrm{Cov}(\langle f,\phi_1 \rangle, \langle f,\phi_2 \rangle), \quad \phi_1, \phi_2 \in \mathcal{D}.
	\] Then define the covariance operator $C_f\,:\, \mathcal{D} \to \mathcal{D}'$ by \[
	\langle C_f \phi_1, \phi_2 \rangle = \mathrm{Cov}(\langle f,\phi_1 \rangle, \langle f,\phi_2 \rangle), \quad \phi_1, \phi_2 \in \mathcal{D},
	\] which is associated with a Schwartz kernel denoted by $K_f(x,y)$ as follows \[
	\langle C_f \phi_1, \phi_2 \rangle=\int_{\mathbb{R^d}}\int_{\mathbb{R^d}} K_f(x,y) \phi_1(x)\phi_2(y) \,\mathrm{d}x\mathrm{d}y, \mbox{ and }
	K_f(x,y)=\mathbb{E}[(f(x)-\mathbb{E}f(x))(f(y)-\mathbb{E}f(y))].
	\]
	
	In this paper, the random source is assumed to be characterized by a special class of generalized Gaussian random functions as follows.
	\begin{Def}
		A generalized Gaussian random function $f$ with zero expectation is called \textbf{microlocally isotropic of order} $m \in \mathbb{R}$ in the domain $D \subset \mathbb{R}^d$, if $\text{supp}\, f \subset \subset D$ for almost surely $\omega \in \Omega$ and its covariance operator $C_f$ is a classical pseudo-differential operator with the principal symbol $h(x) \big| \xi \big|^{-m}$ where $0 \le h \in C^\infty_0(\mathbb{R}^d)$ and $\text{supp} \, h \subset \subset D$.
	\end{Def} 

The smooth function $h$ is called the micro-correlation strength of the random function $f$. Let $c(x,\xi)$ be the symbol of $C_f$. The operator can be represented by \[
	C_f(\phi)(x)=(2 \pi)^{-d} \int_{\mathbb{R}^d} e^{i x \cdot \xi} c(x, \xi) \hat{\phi}(\xi) \,\mathrm{d}\xi, \quad \phi \in \mathcal{D},
	\] where $\hat{\phi}(\xi)$ stands for the Fourier transform of $\phi$ defined by $
	\hat{\phi}(\xi)=\mathcal{F}\phi(\xi):=\int_{\mathbb{R}^d} e^{-ix \cdot \xi} \phi(x) \,\mathrm{d}x.$  
The kernel $K_f$ can be represented as an oscillatory integral of the form \begin{equation}\label{eq2.1}
		K_f(x,y)=(2\pi)^{-d} \int_{\mathbb{R}^d} e^{i(x-y) \cdot \xi} c(x, \xi) \,\mathrm{d}\xi.
	\end{equation}
		
	We introduce the following regularity result in  \cite{li2021inverse} for microlocally isotropic Gaussian random functions.
	\begin{Le}\label{le2.1}
		Let $f$ be a microlocally isotropic Gaussian random function of order $m$. Then $f \in W^{\frac{m-d}{2}-\epsilon,p}(D)$ almost surely for all $\epsilon>0$ and $1<p<\infty$.
	\end{Le}
	By Sobolev embedding theorem, Lemma \ref{le2.1} gives the following corollary.
	\begin{Coro}\label{coro2.1}
		Let $f$ be a microlocally isotropic Gaussian random function of order $m$. \textcolor{black}{Suppose $D$ is a Lipschitz domain such that $\text{ supp}f \subset D$.} For any $a \in \mathbb{N}$ satisfying $d+2a<m \le d+2(a+1)$, we have $f \in C^{a, \gamma}\textcolor{black}{(\bar{D})}$ almost surely with $\gamma \in (0, \frac{m-d}{2}-a)$.
	\end{Coro}

 	\section{Direct problem}\label{direct}
	In this section, we investigate the direct scattering problem. The Green function of the Helmholtz equation $\Phi(x,y)$ has the explicit form \begin{align} \label{eq3.3}
		\Phi(x,y)=\left\{\begin{array}{cc}
			\frac{i}{4} H^{(1)}_0 (k|x-y|),  &  d=2,\\
			\frac{e^{ik|x-y|}}{4\pi |x-y| },  & d=3,
		\end{array}\right.
	\end{align} where $H^{(1)}_0$ is the Hankel function of the first kind with order zero.
	According to Corollary \ref{coro2.1}, for $m>d$, the source $f \in C^{0, \gamma}\textcolor{black}{(\bar{D})} \subset L^2(D)$ with some $\gamma>0$ almost surely. Hence, the scattering problem \eqref{eq3.1}-\eqref{eq3.2} is classical which admits a unique solution $u \in H^2_{loc}(\mathbb{R}^d)$ with the following explicit integral form 
	\begin{align}\label{explicit}
	u= -\int_{\mathbb{R}^d} \Phi(x,y) f(y)\,\mathrm{d}x.
	\end{align}
	However, for $m \le d$, Lemma \ref{le2.1} shows that such source is too rough to exist pointwisely which should be taken as distribution. In this case, the following theorem proved in \cite{li2022far} gives the well-posedness and regularity of the direct problem. 
	\begin{Thm}\label{lax}
		Let $f$ be a microlocally isotropic Gaussian random function of order $m \in (d-4, d]$. The problem \eqref{eq3.1}-\eqref{eq3.2} is well-posed and \textcolor{black}{the unique solution can be represented by \eqref{eq3.3} in sense of distributions,} which satisfies $u \in W^{\alpha,q}_{loc}(\mathbb{R}^3)$ almost surely for any $q>1$ and \[
		0<\alpha< \min \left\{2-\frac{d-m}{2},2-\frac{d-m}{2}+d\left(\frac{1}{q}-\frac{1}{2}\right)\right\}.
		\]
	\end{Thm}

For $m>d$, by classical acoustic wave scattering theory we know that the solution of the direct problem admits a \textcolor{black}{pointwise} explicit integral representation \eqref{explicit}.
However, for $m\leq d$, we only know that \eqref{explicit} holds in sense of distributions. In the rest of this section, we show that \eqref{explicit} holds
pointwisely, which implies $u \in L^\infty(\mathbb{R}^d)$.

The following lemma is useful in the subsequent analysis. 
	\begin{Le}\label{le3.1}
		Supposing that $p_1,p_2 \in (1,d)$ and $p_1+p_2 >d$, we have that the inequalities \begin{align*}
			\int_D \frac{1}{|y-x|^{p_1}|y-z|^{p_2}}\,\mathrm{d}y \lesssim \frac{1}{|x-z|^{p_1+p_2-d}}
		\end{align*} and
\begin{align*}
			\int_D \frac{|\log|z-y||}{|y-x|^{p_1}|y-z|^{p_2}}\,\mathrm{d}y \lesssim \frac{1+|\log|x-z||}{|x-z|^{p_1+p_2-d}}
		\end{align*}
 hold for $x,z \in \mathbb{R}^d$.
	\end{Le}
	\begin{proof}[Proof]
		Without loss of gengerality, assume $z=0$. Letting $y'=\frac{y}{|x|}$ we obtain \begin{align*}
			\int_D \frac{1}{|y-x|^{p_1}|y|^{p_2}}\,\mathrm{d}y \le \int_{\mathbb{R}^d} \frac{|x|^{d-p_1-p_2}}{|y'-\hat{x}|^{p_1}|y'|^{p_2}}\,\mathrm{d}y' \lesssim \frac{1}{|x|^{p_1+p_2-d}},
		\end{align*} which proves the first inequality. The second inequality can be proven similarly.
	\end{proof}

	\begin{Le} \label{le3.2}
		For any $x \in \mathbb{R}^d$ we have $\Phi(x,\cdot) \in W^{1+\mu,p}_{loc}(\mathbb{R}^d)$ with $\mu \in (0,1)$, $1<p <\frac{d}{\mu+d-1}$.
	\end{Le}
		
	\begin{proof} 
		Consider the Slobodeckij semi-norm \[
		|u|^p_{W^{\mu,p}(D)} := \int_D \int_D \frac{|u(x)-u(y)|^p}{|x-y|^{p\mu+d}} \,\mathrm{d}x\,\mathrm{d}y.
		\] The norm of $W^{1+\mu,p}(D)$ can be expressed by \[
		\|u\|^p_{W^{1+\mu,p}(D)}:=\|u\|^p_{W^{1,p}(D)} + \sum_{i=1}^d |\partial_i u|^p_{W^{\mu,p}(D)}.
		\] Without loss of generality, we assume  $x \in D \subset \subset \mathbb{R}^d$.
				Firstly we consider the case when $d=3$. Notice \[
		\Phi(x,y)=\frac{e^{ik|x-y|}}{4\pi |x-y| },\quad \partial_i\Phi(x,y)=e^{ik|x-y|}\frac{ik(y_i-x_i)|x-y|-(y_i-x_i)}{4\pi|x-y|^3},
		\] which gives $\Phi(x,\cdot) \in W^{1,p}(D)$ with $1< p <\frac{3}{2}$. Therefore, we only need to show \begin{align*}
			|\partial_j \Phi(x,y)|^p_{W^{\mu,p}(D)} < \infty, \quad j=1,2...d,
		\end{align*} for $0<\mu<1 $ and $1<p< \frac{3}{2+\mu}$. 
		Clearly, we only need to consider the term which has higher singularity. Hence, it suffices to show  \[
		e^{ik|x-y|}\frac{y_j-x_j}{|x-y|^3} \in W^{\mu,p}(D)
		\]for $0<\mu<1 $ and $1<p< \frac{3}{2+\mu}$. Using the fractional Leibniz rule in \cite{grafakos2014kato} gives \begin{align} 
			\left|e^{ik|x-y|}\frac{y_j-x_j}{|x-y|^3}\right|_{W^{\mu,p}(D)} &\lesssim \|e^{ik|x-y|}\|_{L^\infty(D)}\left|\frac{y_j-x_j}{|x-y|^3}\right|_{W^{\mu,p}(D)} + |e^{ik|x-y|}|_{W^{\mu,\infty(D)}} \left\|\frac{y_j-x_j}{|x-y|^3}\right\|_{L^{p}(D)}. \label{eq3.4}
		\end{align} It is straightforward to verify $e^{ik|x-y|} \in W^{1,\infty}(D)$, which implies \begin{equation}\label{eq3.5}
			\|e^{ik|x-y|}\|_{L^\infty(D)}<\infty,\quad  |e^{ik|x-y|}|_{W^{\mu,\infty(D)}}< \infty.
		\end{equation} Combining \eqref{eq3.4}--\eqref{eq3.5}, we only need to prove \[
		\frac{x_i-y_i}{|x-y|^3} \in W^{\mu,p}(D)
		\] for $0<\mu<1 $ and $1<p< \frac{3}{2+\mu}$.
		It is easy to verify \[
		\frac{x_i-y_i}{|x-y|} \in W^{1,p'}(D),
		\] where $1<p'<3$. 
		Applying Sobolev embedding theorem yields \begin{equation} \label{eq3.6}
			\frac{x_i-y_i}{|x-y|} \in W^{\mu,q}(D),
		\end{equation} where $0<\mu<1$ and $q<\frac{3}{\mu}$. 
		Next we will show that \begin{equation} \label{eq3.7}
			\frac{1}{|x-y|^2} \in W^{\mu,p}(D) 
		\end{equation} for $0<\mu<1$ and $1<p<\frac{3}{2+\mu}$. To this end, direct calculations imply \begin{align}
			&\int_D \int_D \frac{\left| |x-y|^{-2}-|x-z|^{-2} \right|^p}{|y-z|^{p\mu+3}} \,\mathrm{d}y\mathrm{d}z \lesssim \int_D \int_D \frac{(|x-z|^p+|x-y|^p)|z-y|^p}{|y-z|^{p\mu+3}|x-z|^{2p}|x-y|^{2p}} \,\mathrm{d}y\mathrm{d}z \notag \\  = & \int_D \int_D |y-z|^{-(p(\mu-1)+3)}|x-y|^{-p}|x-z|^{-2p}\,\mathrm{d}y\mathrm{d}z  +\int_D \int_D |y-z|^{-(p(\mu-1)+3)}|x-z|^{-p}|x-y|^{-2p}\,\mathrm{d}y\mathrm{d}z. \label{eq3.8}
		\end{align}
		Applying Lemma \ref{le3.1} gives \begin{align} \label{eq3.9}
			\int_D \int_D |y-z|^{-(p(\mu-1)+3)}|x-y|^{-p}|x-z|^{-2p}\,\mathrm{d}y\mathrm{d}z \lesssim \int_D \frac{1}{|x-z|^{2p+\mu p}}< \infty,
		\end{align} when $0<\mu<1 $ and $p<\frac{3}{2+\mu}$. Inserting \eqref{eq3.9} into \eqref{eq3.8} gives \eqref{eq3.7}. 
		
		Applying fractional Leibniz rule yields \begin{align}
			\left|\frac{x_i-y_i}{|x-y|^3}\right|_{W^{\mu,p}(D)} \lesssim \left\|\frac{x_i-y_i}{|x-y|}\right\|_{L^\infty(D)}\left|\frac{1}{|x-y|^2}\right|_{W^{\mu,p}(D)} + \left|\frac{x_i-y_i}{|x-y|}\right|_{W^{\mu, q}(D)}\left|\frac{1}{|x-y|^2}\right|_{L^{q'}(D)}, \label{eq3.10}
		\end{align} where $\frac{1}{q'}+\frac{1}{q}=\frac{1}{p}$. We have known that \begin{equation} \label{eq3.11}
			\left|\frac{1}{|x-y|^2}\right|_{W^{\mu,p}(D)}<\infty,
		\end{equation} for $0<\mu<1$, $1<p<\frac{3}{2+\mu}$ and\begin{equation} \label{eq3.12}
			\left|\frac{x_i-y_i}{|x-y|}\right|_{W^{\mu, q}(D)} < \infty
		\end{equation} for $0<\mu<1$, $1<q<\frac{3}{\mu}$. Obviously, it can be verified \begin{equation} \label{eq3.13}
			\left|\frac{1}{|x-y|^2}\right|_{L^{q'}(D)}< \infty
		\end{equation} with $1<q'<\frac{3}{2}$. Hence, we can choose $1<q<\frac{3}{\mu}$ and $1<q'<\frac{3}{2}$ which gives \[
		p=\frac{1}{q^{-1}+q'^{-1}}< \frac{3}{2+\mu}.
		\] Then by combining \eqref{eq3.10}--\eqref{eq3.13} we arrive at \begin{align*}
			\left|\frac{x_i-y_i}{|x-y|^3}\right|_{W^{\mu,p}(D)} < \infty,
		\end{align*} which completes the proof when $d=3$. 
		
		When $d=2$, the discussion is analogous as $d=3$ by noticing the asymptotic relations \[
		H^{(1)}_0(t)=\frac{2i}{\pi}\log{\frac{t}{2}}+O(1)
		\] and \[
		H^{(1)}_0(t)=-\frac{2i}{\pi t}+O(t)
		\] when $t \to 0$.  We omit it for brevity.
	\end{proof}
	
 \textcolor{black}{Combining Lemma \ref{le2.1} and \ref{le3.2}, we have the following regularity result.}
	\begin{Thm}\label{thm3.2}
	Let $f$ be a microlocally isotropic Gaussian random function of order $m \in (d-4, d]$.
	The solution to the direct scattering problem admits the following representation which holds pointwisely
	\begin{equation} \label{eq3.2.1}
		u(x)=-\int_{\mathbb{R}^d} \Phi(x,y) f(y)\,\mathrm{d}x.
	\end{equation} Moreover, $u \in L^\infty(\mathbb{R}^d)$ almost surely.
	\end{Thm}
	\begin{proof}
	   For any fixed $x \in \mathbb{R}^d$, combining Lemma \ref{le2.1} and \ref{le3.2} gives that \eqref{eq3.2.1} is well-defined, i.e. \begin{align}\label{eq3.2.2}
	       |u(x)| \le \|\Phi(x,\cdot)\|_{W^{1+\mu,p}(D)}\|f\|_{W^{-(1+\mu),p'}(D)}<\infty,
	   \end{align} with \[0<\mu<1,\quad 1<p<\frac{d}{\mu+d-1} \quad \text{\rm and}\quad \frac{1}{p}+\frac{1}{p'}=1.\] Now we show $u \in L^\infty(\mathbb{R}^d)$. Recalling \eqref{eq3.3}, there exists a sufficient large $R_0$ such that $D \subset B_{R_0}  $ and \[
    |\Phi(x,y)| \le C,\quad |\nabla_y\Phi(x,y)| \le C \quad\text{\rm and} \quad |\nabla_y^2\Phi(x,y)| \le C,\quad x \in \mathbb{R}^d\backslash B_{R_0},\, y\in D, \] where $C$ is a constant independent of $x$ and $y$. Hence for any $ x \in \mathbb{R}^d\backslash B_{R_0}$, there holds \begin{align} \label{eq3.2.3}
        \|\Phi(x,\cdot)\|_{W^{1+\mu,p}(D)} \le \|\Phi(x,\cdot)\|_{W^{2,p}(D)} \le C,
    \end{align}where $C$ is a constant independent of $x$. For any $x \in B_{R_0}$, we have \begin{align}\label{eq3.2.4}
        \|\Phi(x,\cdot)\|_{W^{1+\mu,p}(D)} \le \|\Phi(x,\cdot)\|_{W^{1+\mu,p}(B_{R_0})} \le \|\Phi(x,\cdot)\|_{W^{1+\mu,p}(B_{2R_0)}}=\|\Phi(0,\cdot)\|_{W^{1+\mu,p}(B_{2R_0})}.
    \end{align}
    Combining \eqref{eq3.2.2}--\eqref{eq3.2.4} we complete the proof.
	\end{proof}
	
It can be verified that using the regularity result in Theorem \ref{thm4.1} and Sobolev embedding theorem one can only have $u \in L^\infty(\mathbb{R}^d)$ when $m>2(d-2)$. Hence, the above theorem enhances the regularity result in Theorem \ref{thm4.1}.
 
	\section{Inverse problem using near-field data}\label{near}
	 
	In this section, we derive the stability estimate by near-field data. Denote $a(x,\xi)=c(x,\xi)-h(x)|\xi|^{-m}$.
	 Assume that $f$ satisfies the following assumption. 
	
	\emph{Assumption (A). The random source $f$ is a real-valued microlocally isotropic Gaussian random function of order $m>d-1$. The covariance operator has the symbol $c(x,\xi)$ with the principal symbol $h(x)|\xi|^{-m}$ satisfying (i) $|h(x)| \le M$ for $x \in D$; (ii) $|c(x, \xi)| \le M(1+|\xi|)^{-m}$ for $\xi \in \mathbb{R}^d$ and $x \in D$; (iii) $|a(x, \xi)| \le M |\xi|^{-(m+1)}$ for $|\xi| \ge 1$ and $x \in D$}. Here $M>0$ stands for a constant.

	We first give a bound with a high frequency tail for the Fourier transform of $h$ \textcolor{black}{in terms of the source function.}
	Recalling the definitions in Section \ref{pre}, the covariance of $\hat{f}$ can be expressed by 
	\begin{align*}
		\mathbb{E}(\hat{f}(\xi)\hat{f}(\eta)) &=\mathbb{E}\left(\int_{B_R}\int_{B_R} f(x)f(y) e^{-i \xi \cdot x}e^{-i \eta \cdot y}\,\mathrm{d}x\mathrm{d}y\right) =\int_{B_R}\int_{B_R} K_f(x,y) e^{-i \xi \cdot x}e^{-i \eta \cdot y}\,\mathrm{d}x\mathrm{d}y \\ &=(2\pi)^{-d}\int_{B_R}\int_{B_R}\int_{\mathbb{R}^d}e^{-i (\xi \cdot x + \eta \cdot y)}e^{i(x-y)\cdot \zeta}c(x,\zeta)\,\mathrm{d}\zeta\mathrm{d}x\mathrm{d}y\\ &=\int_{B_R}\int_{\mathbb{R}^d}e^{i x \cdot (\zeta-\xi)}c(x,\zeta) \delta(\zeta+\eta)\,\mathrm{d}\zeta\mathrm{d}x = \int_{B_R}e^{-i x \cdot (\eta+\xi)} c(x,  -\eta) \,\mathrm{d}x \\ &=\int_{B_R}e^{-i x \cdot (\eta+\xi)} h(x)|\eta|^{-m} \,\mathrm{d}x + MR^dO(|\eta|^{-(m+1)}) \notag\\ &=|\eta|^{-m}\hat{h}(\xi+\eta) + MR^dO(|\eta|^{-(m+1)}).
	\end{align*} Therefore, we obtain \[
	\hat{h}(\xi+\eta)=|\eta|^{m}\mathbb{E}(\hat{f}(\xi)\hat{f}(\eta))+ MR^d O\left(\frac{1}{|\eta|}\right).
	\] Taking $\xi=k \theta_1$ and $\eta=k \theta_2$ with $\theta_1,\, \theta_2 \in \mathbb{S}^{d-1}$, one arrives at  \begin{equation} \label{eq4.3}
		\hat{h}(k(\theta_1+\theta_2))=|k|^{m}\mathbb{E}(\hat{f}(k\theta_1)\hat{f}(k\theta_2))+MR^dO\left(\frac{1}{|k|}\right).
	\end{equation} 
	
	\textcolor{black}{Next, we bound the term $\mathbb{E}(\hat{f}(k\theta_1)\hat{f}(k\theta_2))$ by the correlation data on $\partial B_R$.}
	Multiplying the governing equation \eqref{eq3.1} by the plane wave $e^{-ik \theta_j \cdot x}$ and integrating by parts yield \begin{align} \label{eq4.4}
		\int_{B_R} f e^{-ik \theta_j \cdot x}\,\mathrm{d}x=\int_{\partial B_R} \partial_\nu u(x) e^{-ik \theta_j \cdot x}-\partial_\nu(e^{-ik \theta_j \cdot x}) u(x)\,\mathrm{d}x.
	\end{align} Taking expectation gives that \begin{align}
		\mathbb{E}(\hat{f}(k\theta_1)\hat{f} (k \theta_2)) &= \int_{\partial B_R}\int_{\partial B_R} \mathbb{E}(\partial_\nu u(x) \partial_\nu u(y)) e^{-ik(\theta_1 \cdot x+\theta_2 \cdot y)}\,\mathrm{d}s(x) \mathrm{d}s(y) \notag\\ &\quad-\int_{\partial B_R}\int_{\partial B_R} \mathbb{E}(\partial_\nu u(x)  u(y)) e^{-ik(\theta_1 \cdot x+\theta_2 \cdot y)}(-ik \theta_2 \cdot \hat{y})\,\mathrm{d}s(x) \mathrm{d}s(y) \notag\\ &\quad-\int_{\partial B_R}\int_{\partial B_R} \mathbb{E}(\partial_\nu u(y)  u(x)) e^{-ik(\theta_1 \cdot x+\theta_2 \cdot y)}(-ik \theta_1 \cdot \hat{x})\,\mathrm{d}s(x) \mathrm{d}s(y) \notag \\ &\quad- \int_{\partial B_R}\int_{\partial B_R} \mathbb{E}( u(x)  u(y)) e^{-ik(\theta_1 \cdot x+\theta_2 \cdot y)}k^2( \theta_2 \cdot \hat{y})( \theta_1 \cdot \hat{x})\,\mathrm{d}s(x) \mathrm{d}s(y) \notag\\ &\quad:= \sum_{j=1}^4 I_j(k, \theta_1, \theta_2), \label{eq4.5}
	\end{align} where $\hat{x}=x/|x|$. By inserting \eqref{eq4.5} into \eqref{eq4.3} and noticing 
	\[
	\textcolor{black}{\{\theta: \theta =\theta_1 + \theta_2, \theta_1, \theta_2\in\mathbb S^{d-1}\} = \{\theta\in\mathbb R^d: |\theta|\leq 2\},}
	\]
	we derive the inequality \begin{align} \label{eq4.6}
		|\hat{h}(\xi)|^2 \lesssim \sup_{\theta_1, \theta_2 \in \mathbb{S}^{d-1}} k^{2m} |\sum_{j=1}^4 I_j(k, \theta_1, \theta_2) |^2 + \frac{M^2R^{2d}}{k^2},
	\end{align} which holds for all $|\xi|\le 2k$. For convenience, denote \[
	\epsilon^2(k, \theta_1, \theta_2):= \sum_{j=1}^4 |I_j(k, \theta_1, \theta_2)|^2
	\] with $k>0$ and $\theta_1, \theta_2 \in \mathbb{S}^{d-1}$. From the equation \eqref{eq3.1} and the Sommerfeld radiation condition \eqref{eq3.2}, it can be verified that \[
	u(x; -k)=\overline{u(x;k)},\quad k \in \mathbb{R},
	\] where the notation $u(x;k)$ is used to exhibit the dependence of the solution on the wavenumber $k$. Therefore, the definition of $\epsilon^2(k, \theta_1, \theta_2)$ can be extended to $\mathbb{C}$ as
	\[
	\epsilon^2(k, \theta_1, \theta_2):= \textcolor{black}{\sum_{j=1}^4 I_j(k, \theta_1, \theta_2) I_j(-k, \theta_1, \theta_2).}
	\] Consider the multi-frequency data characterized as \[
	\epsilon^2= \sup_{k \in [0,K], \theta_1, \theta_2 \in \mathbb{S}^{d-1}} \epsilon^2(k, \theta_1, \theta_2),
	\] where $K>0$ is the upper bound of the frequency. Denote a sectorial domain by $\mathcal{R}=\{z\in \mathbb{C}: |\arg z| < \pi/4\}$.
	In what follows, we show that $\epsilon^2(k,\theta_1,\theta_2)$ is analytic and has an upper bound for $k \in \mathcal{R}$. 
\textcolor{black}{We only consider the term $I_1$ since the discussions for $I_2,I_3,I_4$ are similar. Recalling \eqref{eq3.2.1}, we deduce \begin{align*}
		&I_1(k,\theta_1,\theta_2) \\ &=\int_{\partial B_R}\int_{\partial B_R} \mathbb{E}( u(x)  u(y)) e^{-ik(\theta_1 \cdot x+\theta_2 \cdot y)}\,\mathrm{d}s(x) \mathrm{d}s(y) \\ &= \int_{\partial B_R}\int_{\partial B_R}\int_{D}\int_{D} J(x,y,\tau,t;k) \mathbb{E}(f(\tau)f(t))\,\mathrm{d}\tau\mathrm{d}t\mathrm{d}s(x) \mathrm{d}s(y)\\ &= \int_{\partial B_R}\int_{\partial B_R}\int_{D}\int_{D}  J(x,y,\tau,t;k)K_f(\tau,t)\,\mathrm{d}\tau\mathrm{d}t\mathrm{d}s(x) \mathrm{d}s(y),
	\end{align*} where we denote  \[
	J(x,y,\tau,t;k):=\partial_{\nu(x)} \Phi(x,\tau) \partial_{\nu(y)} \Phi(y,t)e^{-ik(\theta_1 \cdot x+\theta_2 \cdot y)}.
	\]  Obviously, $J$ is analytic with respective to $k \in \mathcal{R}$. In order to show $I_1$ is also analytic,}  we shall verify the following estimates \begin{align} \label{eq4.7}
		&\int_{\partial B_R}\int_{\partial B_R}\int_{D}\int_{D} |K_f(\tau,t)J(x,y,\tau,t;k)|\,\mathrm{d}\tau\mathrm{d}t\mathrm{d}s(x) \mathrm{d}s(y) < \infty\\
 \label{eq4.8}
		&\int_{\partial B_R}\int_{\partial B_R}\int_{D}\int_{D} |K_f(\tau,t)\partial_k J(x,y,\tau,t;k)|\,\mathrm{d}\tau\mathrm{d}t\mathrm{d}s(x) \mathrm{d}s(y) < \infty
	\end{align}
	which hold uniformly with respect to $k\in \mathcal{R}$. As a consequence, the derivative can be taken under the integral.
	We only prove \eqref{eq4.7}, since \eqref{eq4.8} can be proven similarly.  It is necessary to discuss the singularity of $K_f(\tau,t)$. Inspired by \cite{caro2019inverse}, we show in the following lemma that the kernel $K_f(x,y)$ can be represented as the sum of a singular part and a continuous remainder which is bounded under \emph{assumption (A)}.  
	\begin{Le}\label{le4.1}
		Let the random function $f$ satisfy assumption (A). The covariance function $K_f$ has the following form:
		
		(i)  If $a<m-(d-1)\le a+1$ with $a=1,2,3...$,  \[
		K_f(x, y)= ch(x)|x-y|^{m-d} +F_m(x,y)
		\] where $c$ is a constant dependent on $m,d$ and $F_m(x,y) \in C^{a,\alpha}(\mathbb{R}^d \times \mathbb{R}^d)$ with $\alpha \in (0, m-d-a+1)$.
		
		(ii)
		If $d-1<m< d$,  \[
		K_f(x, y)= ch(x)|x-y|^{m-d} +F_m(x,y)
		\] where $c$ is a constant dependent on $m,d$ and $F_m(x,y) \in C^{0,\alpha}(\mathbb{R}^d \times \mathbb{R}^d)$ with $\alpha \in (0, m-d+1)$. 
		
		(iii)
		If $m=d$,  \[
		K_f(x, y)= ch(x)\log{|x-y|} +F_m(x,y)
		\] where $c$ is a constant dependent on $d$ and $F_m(x,y) \in C^{0,\alpha}(\mathbb{R}^d \times \mathbb{R}^d)$ with $\alpha \in (0, 1)$. 
		
		For all of the above three cases, we have \[
		\|F_m\|_{L^\infty(D \times D)} \lesssim M. 
		\]
		
	\end{Le}
	\begin{proof}[Proof]
	\textcolor{black}{We first consider the case $a<m-(d-1) \le a+1$.}
		Choose a radially symmetric cut-off function $\chi(\xi) \in C^\infty_0(\mathbb{R}^d)$ such that $\chi(\xi)=1$ when $|\xi| \le 1$.   Recalling \eqref{eq2.1} , we deduce \begin{align}
			\langle K(x,x-\cdot),  \phi \rangle &= \int_{\mathbb{R}^d} \mathcal{F}^{-1}(\chi c(x, \cdot))(y) \phi(y)\,\mathrm{d}y  +\int_{\mathbb{R}^d} \mathcal{F}^{-1}((1-\chi) a(x, \cdot))(y) \phi(y)\,\mathrm{d}y \notag\\ &\quad+ \int_{\mathbb{R}^d} (1-\chi(\xi))h(x)|\xi|^{-m}  (\mathcal{F}^{-1}\phi)(\xi)\,\mathrm{d}\xi \label{eq4.14}
		\end{align} with the test function $\phi \in \mathcal{D}$. Notice that $v_1(x,y) :=\mathcal{F}^{-1}(\chi c(x,\cdot))(y) \in \mathcal{S}(\mathbb{R}^d \times \mathbb{R}^d)$ with $supp \, v \subset D \times \mathbb{R}^d$ and the bound \begin{align} \label{eq4.15}
			\|v_1(x,y)\|_{L^\infty(\mathbb{R}^d \times \mathbb{R}^d)} \lesssim M \int_{supp\, \chi} |\chi| (1+|\xi|)^{-m} \,\mathrm{d}\xi \lesssim M. 
		\end{align} Denote $v_2(x,y)=\mathcal{F}^{-1}((1-\chi)a(x,\cdot))(y)$. Obviously $v_2$ is smooth and compactly supported with respect to $x$. Recalling $1-\chi$ vanishes when $|\xi| \le 1$, then \emph{assumption (A)} implies \[
		|(1-\chi(\xi))a(x,\xi)| \lesssim  M |\xi|^{-(m+1)}.
		\] Take $p,\,q$ such that $2 \le p \le \infty$ and $\frac{1}{p}+\frac{1}{q}=1$. Applying \emph{assumption (A)} and Hausdorff-Young inequality gives the inequality \begin{align}
			\|v_2(x, \cdot)\|_{W^{s,p}(\mathbb{R}^d)} &=\|(I-\Delta)^{\frac{s}{2}}v_2(x, \cdot)\|_{L^{p}(\mathbb{R}^d)}  \lesssim  \|(1+|\cdot|^2)^{s/2}(1-\chi)a(x,\cdot)\|_{L^{q}(\mathbb{R}^d)} \notag \\ &\lesssim M \left(\int_{|\xi| \ge 1} \frac{1}{|\xi|^{(m+1-s)q}}\,\mathrm{d}\xi \right)^{\frac{1}{q}} \lesssim M. \notag
		\end{align} The above inequality holds uniformly with respect to $x$ if and only if $(m+1-s)q >n$ which can be satisfied by choosing $s>0$ and $s-{d}/{p}<m-(d-1)$. Then using Sobolev embedding theorem gives \textcolor{black}{$ W^{s,p}(\mathbb{R}^d)\subset C^{k_0, \alpha}(\mathbb{R}^d)$} with $k_0 \in \mathbb{N}, \alpha \in (0,1]$ such that $k_0 +\alpha= s-d/p$. When $a<m-(d-1) \le a+1$ with $a \in \mathbb{N}$, we can choose $p$ sufficiently large such that $s$ is close to $m-(d-1)$ which implies $k_0, \alpha$ can be chosen as $k_0=a$ and $0<\alpha<m-(d-1)-a $. Furthermore, one has $v_2 \in C^{a,\alpha}(\mathbb{R}^d \times \mathbb{R}^d)$ and \begin{equation}  \label{eq4.16}\|v_2(x,y)\|_{L^\infty(\mathbb{R}^d \times \mathbb{R}^d)} \le 
			\|v_2(x, \cdot)\|_{C^{a,\alpha}(\mathbb{R}^d)} \le  \|v_2(x, \cdot)\|_{W^{s,p}(\mathbb{R}^d)} \lesssim M.
		\end{equation} The third term in \eqref{eq4.14} can be rewritten as \begin{align*}
			h(x) \int_{\mathbb{R}^d} \mathcal{F}^{-1}((1-\chi)|\cdot|^{-m})(y) \phi(y) \,\mathrm{d}y.
		\end{align*}
		Now we claim when $m>d$, \begin{equation} \label{eq4.17}
			\mathcal{F}^{-1}((1-\chi)|\cdot|^{-m})(y)=c_1|y|^{m-d} + g_1(y),
		\end{equation} where $c_1$ is a constant only dependent on $m,d$ and $g_1(y) \in C^\infty(\mathbb{R}^d)$. In fact, direct calculation gives
		\begin{align*}
			\mathcal{F}^{-1}((1-\chi)|\cdot|^{-m})(y) &=|y|^{m-d}\mathcal{F}^{-1}((1-\chi)|\cdot|^{-m})(\hat{y}) + (2 \pi)^{-d} \int_{\mathbb{R}^d} e^{i y \cdot \xi} |\xi|^{-m}(\chi(|y|\xi)-\chi(\xi)) \,\mathrm{d}\xi. 
		\end{align*} Notice that $m>d$ and choose $\chi$ to be a radially symmetric function. We have that \[
		c_1:=\mathcal{F}^{-1}((1-\chi)|\cdot|^{-m})(\hat{y})
		\] is a finite constant. As $\chi \in C_0^\infty$, we have that \[
		g_1(y):= (2 \pi)^{-d}\int_{\mathbb{R}^d} e^{i y \cdot \xi} |\xi|^{-m}(\chi(|y|\xi)-\chi(\xi)) \,\mathrm{d}\xi
		\] is smooth and bounded with respect to $y$.
		In conclusion, when $m>d$, by combining \eqref{eq4.14}--\eqref{eq4.17} we have  \[
		K(x, x-y)= c_1 h(x) |y|^{m-d} +v_1(x,y)+v_2(x,y)+ h(x)g_1(y),
		\] which gives \[
		K_f(x, y)= c_{m,d}h(x)|x-y|^{m-d} +F_m(x,y), \quad \|F_m\|_{L^\infty(\mathbb{R}^d \times \mathbb{R}^d)} \lesssim M 
		\] with \[
		F_m(x,y):=v_1(x,x-y)+v_2(x,x-y)+ h(x)g(x-y) \in C^{a,\alpha}(\mathbb{R}^d \times \mathbb{R}^d)
		\] for $a<m-(d-1) \le a+1$.
		
		Then we consider the case $d-1<m<d$. Denote by $\Gamma$ the Gamma function  $ \Gamma(\beta)=\int^\infty_0 t^{\beta-1} e^{-t} \,\mathrm{d}t.
		$ One has the identity \begin{align} \label{eq4.18}
			|\xi|^{-m}=\frac{2^{-m/2}}{\Gamma({m}/{2})} \int_0^\infty t^{m/2-1} e^{{-t|\xi|^2}/{2}}\,\mathrm{d}t.
		\end{align} One also has  \begin{equation}\label{eq4.19}
			\mathcal{F}(e^{-t|\xi|^2/2})(y)=(2\pi)^{d/2} t^{-d/2} e^{-|y|^2/(2t)}.
		\end{equation} Applying \eqref{eq4.18}--\eqref{eq4.19} gives \begin{align}
			&\int_{\mathbb{R}^d} (1 -\chi(\xi))h(x)|\xi|^{-m}  (\mathcal{F}^{-1}\phi)(\xi)\,\mathrm{d}\xi =c_2 h(x) \int_{\mathbb{R}^d}|y|^{m-d} \mathcal{F}((1-\chi)\mathcal{F}^{-1}\phi)(y)\,\mathrm{d}y\notag\\ &=  c_2 h(x)\left(\int_{\mathbb{R}^d}|y|^{m-d} \phi(y)\,\mathrm{d}y+ \int_{\mathbb{R}^d}|y|^{m-d}\mathcal{F}(\chi\mathcal{F}^{-1}\phi)(y)\,\mathrm{d}y\right)\notag\\ & = c_2h(x)\int_{\mathbb{R}^d}[|y|^{m-d}+(|\cdot|^{m-d}\ast \mathcal{F}^{-1}\chi)(y)]\phi(y)\,\mathrm{d}y, \label{eq4.20}
		\end{align} where $c_{2}=\frac{2^{-m+d/2}}{\Gamma(m/2)}\Gamma((d-m)/2)$. Moreover, $
		g_2(y):=c_2[|\cdot|^{m-d}\ast \mathcal{F}^{-1}\chi)(y)] 
		$  is smooth. Combining \eqref{eq4.14}--\eqref{eq4.16} and \eqref{eq4.20} yields \[
		K_f(x, y)= c_{2}h(x)|x-y|^{m-d} +F(x,y), \quad \|F\|_{L^\infty(D \times D)} \lesssim M 
		\] with \[
		F(x,y):=v_1(x,x-y)+v_2(x,x-y)+ h(x)g(x-y) \in C^{0,\alpha}(\mathbb{R}^d \times \mathbb{R}^d).
		\] for $d-1<m<d$. 
		
		At last, consider the case when $d=m$. Rewrite the right-hand side of \eqref{eq4.14} as \begin{align*}
			\int_{\mathbb{R}^d} (1 &-\chi(\xi))h(x)|\xi|^{-m}  (\mathcal{F}^{-1}\phi)(\xi)\,\mathrm{d}\xi \\ &=\frac{2^{-m+d/2}}{\Gamma(m/2)} \Gamma((d-m)/2+1)h(x)\int_{\mathbb{R}^d}\frac{|y|^{m-d}-1}{(d-m)/2}\mathcal{F}((1-\chi)\mathcal{F}^{-1}\phi)(y)\,\mathrm{d}y \\ &+\frac{2^{-m+d/2}}{\Gamma(m/2)} \Gamma((d-m)/2)h(x)\int_{\mathbb{R}^d}\mathcal{F}((1-\chi)\mathcal{F}^{-1}\phi)(y)\,\mathrm{d}y.
		\end{align*}
		Letting $m \to d$ in, we obtain \begin{align} 
			&\int_{\mathbb{R}^d} (1 -\chi(\xi))h(x)|\xi|^{-n}  (\mathcal{F}^{-1}\phi)(\xi)\,\mathrm{d}\xi  \notag \\ &= -c_d h(x)\int_{\mathbb{R}^d} \log{|y|}\mathcal{F}((1-\chi)\mathcal{F}^{-1}\phi)(y)\,\mathrm{d}y +\tilde{c}_dh(x)\int_{\mathbb{R}^d}\mathcal{F}((1-\chi)\mathcal{F}^{-1}\phi)(y)\,\mathrm{d}y \label{eq4.21}
		\end{align} with constants $c_d, \tilde{c}_d>0$. Notice that \begin{align}
			g_3(y)=-c_d[\log{|\cdot|} \ast \mathcal{F}^{-1}\chi](y)+\tilde{c}_d\left\{1-\int_{\mathbb{R}^d}\mathcal{F}^{-1}\chi(y)\,\mathrm{d}y\right\}  \label{eq4.22}
		\end{align} is smooth. 
		Combining \eqref{eq4.14}--\eqref{eq4.16} and \eqref{eq4.21}--\eqref{eq4.22} yields \[
		K_f(x,y)=c_d h(x)\log|x-y|+F_m(x,y),\quad\|F\|_{L^\infty(D \times D)} \lesssim M 
		\] with \[
		F_m(x,y):=v_1(x,x-y)+v_2(x,x-y)+ h(x)g_3(x-y) \in C^{0,\alpha}(\mathbb{R}^d \times \mathbb{R}^d)
		\] for $m=d$.
	\end{proof}
	
	\textcolor{black}{In what follows, we show that $I_1$ is bounded with respect to $k\in\mathcal{R}$ and similar arguments apply to $I_2, I_3, I_4$.}
	 We consider the following two situations.
	
	\emph{Case 1.} Let $m>d-1$ and $m \neq d$. We have from Lemma \ref{le4.1} that  \[
		K_f(x, y)= ch(x)|x-y|^{m-d} +F_m(x,y),
		\] which yields that the kernel is continuous or weakly singular. Therefore, \eqref{eq4.7} holds for $k \in \mathcal{R}$ and thus \textcolor{black}{$I_1$} is analytic. Moreover, when $k=k_1+ik_2 \in \mathcal{R}$, we further obtain the following bound  
\begin{align}
		|I_1| 
		&\lesssim e^{2k_1}\int_{\partial B_R}\int_{\partial B_R}\int_{D}\int_{D}|\partial_{\nu(x)}\Phi(x,\tau)\partial_{\nu(y)}\Phi(y,t)| ({h(\tau)} |\tau-t|^{m-d}+M)\,\mathrm{d}\tau\mathrm{d}t\mathrm{d}s(x) \mathrm{d}s(y) \notag\\ &\lesssim e^{2k_1} M \int_{\partial B_R}\int_{\partial B_R}\int_{D}\int_{D}|\partial_{\nu(x)}\Phi(x,\tau)\partial_{\nu(y)}\Phi(y,t)|  (|\tau-t|^{m-d}+1)\,\mathrm{d}\tau\mathrm{d}t\mathrm{d}s(x) \mathrm{d}s(y). \label{eq4.9}
	\end{align} 

\textcolor{black}{To proceed, noting that the Green function $\Phi$ takes different forms for $d=2$ and $d=3$, we discuss the following two cases.}

If $d=3$, a direct calculation gives \begin{equation} \label{eq4.10}
		|\partial_{\nu(y)}\Phi(x,y)| \le \frac{|k_1|e^{k_1|x-y|}}{4\pi|x-y|}+\frac{e^{k_1|x-y|}}{4\pi|x-y|^2}.
	\end{equation} We show that
\begin{align}\label{eqq4.1}
|I_1| \lesssim e^{2(2R+1)|k_1|} M R^4 (1+|k_1|)^2.
\end{align}
In fact, we only need to estimate the term with highest singularity. To this end, with the help of Lemma \ref{le3.1} we have the following estimate \begin{align}
&\int_{\partial B_R}\int_{\partial B_R}\int_{D}\int_{D}\frac{1+|t-\tau|^{m-d}}{|x-\tau|^2 |y-t|^2}\,\mathrm{d}\tau\mathrm{d}t\mathrm{d}s(x) \mathrm{d}s(y) \lesssim \int_{\partial B_R}\int_{\partial B_R}\int_{D}\frac{1}{|x-\tau|^2|y-\tau|^2}\,\mathrm{d}\tau \mathrm{d}s(x)\mathrm{d}s(y) \notag\\  &\lesssim  \int_{\partial B_R}\int_{\partial B_R}\frac{1}{|x-y|}\mathrm{d}s(x)\mathrm{d}s(y) \lesssim R^4.\label{new1}
\end{align} For other parts in right-hand side of \eqref{eq4.9}, we have simialr estimates. Combining \eqref{eq4.9} and \eqref{new1} yields \eqref{eqq4.1}.
 
 When $d=2$, by the following integral form of the Hankel function \cite{finco2006p}
	\begin{equation} \label{eq4.12}
		H_0^{(1)}(z) = C e^{iz} \int_{0}^\infty e^{-s}s^{-1/2}(s/2-iz)^{-1/2}\,\mathrm{d}s,
	\end{equation} we obtain \begin{align*}
		  |\partial_{\nu(y)}\Phi(x,y)| \lesssim \frac{|k_1|e^{k_1|x-y|}}{|k_1(x-y)|^{\frac{1}{2}}} + \frac{e^{k_1|x-y|}}{|k_1|^\frac{1}{2}|x-y|^{\frac{3}{2}}}, \quad k \in \mathcal{R},\, x \in D, \, y \in \partial B_R.
	\end{align*}  
Then in a similar way as the derivation of \eqref{eqq4.1}, we can obtain \begin{align}\label{eqq4.2}|I_1| \lesssim |k_1|^{-1} e^{2(2R+1)|k_1|} M R^2 (1+|k_1|)^2. 
\end{align} Combining \eqref{eqq4.1} and \eqref{eqq4.2} gives the estimate 
\begin{equation} \label{eqq4.3}
|I_1| \lesssim e^{2(2R+1)|k_1|} M R^{2d-2} |k_1|^{d-3} (1+|k_1|)^2.
\end{equation}

	\textcolor{black}{\emph{Case 2.} When $m=d$ , Lemma \ref{le4.1} shows $K_{f}(\tau,t)=c h(\tau)  \log{|\tau-t|}+F_m(\tau,t)$ which is still weakly singular. Therefore, we can apply the second inequality of Lemma \ref{le3.1} to verify $I_1$ satisfies the inequalities \eqref{eqq4.3}}.
	
	\textcolor{black}{Combining the above arguments, we arrive at the following lemma which provides the analyticity and boundedness of the data with respect to $k\in\mathcal{R}$.}
	\begin{Le} \label{le4.2}
		Suppose that the random function $f$ satisfies assumption (A). We have that $I_j(k,\theta_1,\theta_2), j= 1, ... , 4,$ is analytic with respective to $k \in \mathcal{R}$. Furthermore, the following estimates \textcolor{black}{\begin{align*}
			|I_1(k,\theta_1,\theta_2)| &\lesssim e^{2(2R+1)|k_1|} M R^{2d-2} |k_1|^{d-3} (1+|k_1|)^2,\\
			|I_2(k,\theta_1,\theta_2)| &\lesssim e^{2(2R+1)|k_1|} M R^{2d-2} |k_1|^{d-2} (1+|k_1|),\\
			|I_3(k,\theta_1,\theta_2)| &\lesssim e^{2(2R+1)|k_1|} M R^{2d-2} |k_1|^{d-2} (1+|k_1|),\\
			|I_4(k,\theta_1,\theta_2)| &\lesssim e^{2(2R+1)|k_1|} M R^{2d-2} |k_1|^{d-1} ,
		\end{align*} hold for $\textcolor{black}{k=k_1+ik_2 \in \mathcal{R}}$.}
	\end{Le}
	\begin{Rem}\textcolor{black}{
		(i) Conclusions like Lemma \ref{le4.2} also hold for $I_j(-k, \theta_1, \theta_2)$, $j=1,2,3,4$ after analogous discussions.}
		
		(ii) \textcolor{black}{We have to use Lemma \ref{le3.1} to derive inequality \eqref{new1}.
		Indeed, for $x,y \in \partial B_R$ and $t,\tau \in D$, by a direction calculation one has \[\int_{\partial B_R}\int_{\partial B_R}\int_{D}\int_{D}\frac{1+|t-\tau|^{m-d}}{|x-\tau|^2 |y-t|^2}\,\mathrm{d}\tau\mathrm{d}t\mathrm{d}s(x) \mathrm{d}s(y) \lesssim R^4/\text{dist}(\partial B_R, D)^4.
\] However, when $\text{dist}(\partial B_R, D)\to 0$, the right-hand side of the above inequality tends to infinity. Based on this reason, we use Lemma \ref{le3.1} to estimate this integral.}
	\end{Rem}
	
	The following unique continuation argument \cite{2016increasing} is useful in the subsequent analysis.
	\begin{Le}\label{le4.3}
		Let $U(z)$ be an analytic function in $\mathcal{R}$ and continuous in $\bar{\mathcal{R}}$. If  \begin{align*}
			\left\{\begin{array}{ccc}
				|U(z)| \le \varepsilon, &   z \in (0,L],\\
				|U(z)| \le V, &   z \in \mathcal{R}, \\
				|U(z)|=0, & z=0,
			\end{array}\right.
		\end{align*} with constants $\varepsilon,L,V>0$, then there exists a function $\mu(z)$ satisfying 
		\begin{align*}
			\left\{\begin{array}{cc}
				\mu(z) \ge 1/2, & z \in (L, 2^{1/4} L), \\ 
				\mu(z) \ge \pi^{-1}((z/L)^4-1)^{-1/2},  &  z \in (2^{1/4}L, \infty)
			\end{array}\right.
		\end{align*} such that \[
		|U(z)| \le V \varepsilon^{\mu(z)},\quad \forall z \in (L, \infty).
		\] 
	\end{Le}
	Combining Lemma \ref{le4.2}--\ref{le4.3} yields the following conclusion. \begin{Le}\label{le4.4}
		Let $f$ satisfy assumption (A). Then we have the estimte \[
		|\epsilon^2(k,\theta_1,\theta_2)| \lesssim K^{-1}M^2 R^{4d-4}\epsilon^{2\mu(k)} e^{(8R+5)k},\quad k \in (K,\infty)
		\] with the function $\mu(k)$ satisfying
		\begin{align}\label{eq4.23}
			\left\{\begin{array}{cc}
				\mu(k) \ge 1/2, & k \in (K, 2^{1/4} K), \\ 
				\mu(k) \ge \pi^{-1}((k/K)^4-1)^{-1/2},  &  k \in (2^{1/4}K, \infty).
			\end{array}\right.
		\end{align} .
	\end{Le}
	\textcolor{black}{\begin{proof}[Proof] 
Denote $U(k)=k \epsilon^2(k,\theta_1,\theta_2)$. We have that $U(0)=0$ and $U(k)$ is analytic and continuous for $k \in \bar{\mathcal{R}}$.
		It follows from Lemma \ref{le4.2} when $k \in\mathcal{R}$, \begin{align*}
			|U(k)| \lesssim |k| \sum_{j=1}^4|I_j(k)|\sum_{l=1}^4|I_l(-k)| \lesssim e^{8R+5|k_1|} M^2 R^{4d-4},
		\end{align*} which yields \begin{align*}
			|e^{-(8R+5)k} U(k)| \lesssim M^2 R^{4d-4}.
		\end{align*} Obviously, for $k \in (0,K]$, one has $|e^{-(8R+5)k} U(k)| \le  \epsilon^2$. Applying Lemma 4.4 to $e^{-(8R+5)k} U(k)$ gives \[
		|e^{-(8R+5)k} U(k)| \lesssim  M^2 R^{4d-4}\epsilon^{2\mu(k)},\quad k \in (K,\infty)
		\] where $\mu(k)$ satisfies \eqref{eq4.23}. Thus, we have \[
		|\epsilon^2(k,\theta_1,\theta_2)| \lesssim K^{-1} M^2 R^{4d-4} \epsilon^{2\mu(k)} e^{(8R+5)k},\quad k \in (K,\infty),
		\]
		which completes the proof.
	\end{proof}} 
	Denote $\mathcal{C}_Q=\{h \in C_0^\infty(D): \|h\|_{H^s(\mathbb{R}^d)} \le Q \}$ with $s>0$. We are now in a position to state the increasing stability result of inverse random source problem. The proof adopts the argument of analytic continuation developed in \cite{zhai2023increasing}.
	\begin{Thm}\label{thm4.1}
		Let the random function $f$ satisfy  assumption (A) and assume $h \in \mathcal{C}_Q$. We have the following stability estimate 
		\begin{align}\label{stability}
			\|h\|^2_{L^2(D)} \lesssim K^{2m+\frac{2d}{d+2s}}\epsilon^2 + (R+1)^{2d+\frac{2}{3}}\frac{Q^2+M^2}{K^{\frac{8s}{3(2s+d)}}E^\frac{s}{2s+d}},
		\end{align} with $E=|\log{\epsilon}|$.
	\end{Thm} \begin{proof}[Proof]
		Without losing generality, we assume $\epsilon < e^{-1}$. Then take \begin{align*}
			A= \left\{\begin{array}{cc}
				\frac{1}{((8R+5)\pi)^{\frac{1}{3}}} K^{\frac{2}{3}}E^{\frac{1}{4}},  & 2^{\frac{1}{4}}((8R+5)\pi)^{\frac{1}{3}}K^{\frac{1}{3}}< E^{\frac{1}{4}}, \\
				K,  & E^{\frac{1}{4}} \le 2^{\frac{1}{4}}((8R+5)\pi)^{\frac{1}{3}}K^{\frac{1}{3}}.
			\end{array}\right.
		\end{align*}
		
		If $2^{\frac{1}{4}}((8R+5)\pi)^{\frac{1}{3}}K^{\frac{1}{3}}< E^{\frac{1}{4}}$, we have $A>2^{\frac{1}{4}}K>K$. Then for $k \in (K,A]$, applying Lemma \ref{le4.4} yields \begin{align*}
			\epsilon^2(k,\theta_1,\theta_2) &\lesssim K^{-1} \textcolor{black}{R^{4d-4}} M^2\epsilon^{2\mu(k)} e^{(8R+5)k} \\  &\lesssim K^{-1} \textcolor{black}{R^{4d-4}} M^2 \exp\left\{(8R+5)A - \frac{2E}{\pi}((k/K)^4-1)^{-\frac{1}{2}}\right\} \\ &\lesssim K^{-1} \textcolor{black}{R^{4d-4}} M^2\exp\left\{(8R+5)\frac{K^{\frac{2}{3}}E^\frac{1}{4}}{((8R+5)\pi)^\frac{1}{3}} - \frac{2E}{\pi}(k/A)^2\right\} \\ &\lesssim K^{-1} R^{4d-4} M^2 \exp\left\{-2\left(\frac{(8R+5)^2}{\pi}\right)^\frac{1}{3}K^\frac{2}{3}E^\frac{1}{2}\left(1-\frac{1}{2}E^{-\frac{1}{4}}\right)\right\}.
		\end{align*} Noticing that $\epsilon<e^{-1}$ implies \[
		1-\frac{1}{2}E^{-\frac{1}{4}}>\frac{1}{2}.
		\]  Then we get \begin{align}
			\epsilon^2(k,\theta_1,\theta_2) &\lesssim K^{-1} M^2 \textcolor{black}{R^{4d-4}} \exp\left\{-(8R+5)^{\frac{2}{3}}K^\frac{2}{3}E^\frac{1}{2}\right\}  \lesssim K^{-1} M^2 R^{4d-4} \frac{1}{((8R+5)^{\frac{2}{3}}K^\frac{2}{3}E^\frac{1}{2})^n}. \label{eq4.24}
		\end{align} Here we have used $e^{-x} \le n !/ x^n$ where $n \in \mathbb{N}$. Combining \eqref{eq4.6} and \eqref{eq4.24} gives the estimte \begin{align*}
			|\hat{h}(\xi)|^2 \lesssim M^2\frac{R^{2d}}{A^2} + K^{-1} M^2 A^{2m} \textcolor{black}{R^{4d-4}} \frac{1}{((8R+5)^{\frac{2}{3}}K^\frac{2}{3}E^\frac{1}{2})^n}, \quad |\xi| \le 2A,
		\end{align*}
		which yields \begin{equation} \label{eq4.25}
			\int_{|\xi| \le 2 A^\gamma}|\hat{h}(\xi)|^2\,\mathrm{d}\xi \lesssim R^{2d}M^2A^{d\gamma-2}+\textcolor{black}{R^{4d-4}}A^{d\gamma+2m}K^{-1} M^2\frac{1}{((8R+5)^{\frac{2}{3}}K^\frac{2}{3}E^\frac{1}{2})^n}.
		\end{equation} 
By the condition $\|h\|_{H^s(\mathbb{R}^d)} \le Q$, we obtain \begin{equation} \label{eq4.26}
			\int_{|\xi| > 2 A^\gamma}|\hat{h}(\xi)|^2\,\mathrm{d}\xi \lesssim \frac{Q^2}{A^{2\gamma s}}.
		\end{equation}  Take $n>m+1$ and $\gamma=\frac{2}{2s+d}$ such that $\frac{n-m}{2}-\frac{1}{4}\gamma d > \frac{1}{4}(2-\gamma d)= \frac{1}{2}\gamma s$.
 Therefore, together with \eqref{eq4.25}--\eqref{eq4.26} we have the inequality \begin{align*}
			\|h\|^2_{L^2(D)} &=\int_{|\xi| \le 2 A^\gamma}|\hat{h}(\xi)|^2\,\mathrm{d}\xi+\int_{|\xi| > 2 A^\gamma}|\hat{h}(\xi)|^2\,\mathrm{d}\xi \lesssim E^{-\frac{s}{2s+d}}\left(\frac{(Q^2+1)(R+1)^{2d+\frac{2}{3}}}{K^\frac{8s}{3(2s+d)}}+\frac{M^2\textcolor{black}{R^{4d-4}}}{K^\beta(1+R)^{\frac{2(m+n)}{3}}}\right), 
		\end{align*} where $\beta=1+\frac{2}{3}n-\frac{2}{3}(2m+\frac{2d}{2s+d})$. Taking \textcolor{black}{$n>\max\{3d-7-m,2m+\frac{1}{2}\}$ and $\beta>\frac{8s}{3(2s+d)}$, we have} \begin{equation}
			\|h\|^2_{L^2(D)}\lesssim (R+1)^{2d+\frac{2}{3}} \frac{Q^2+M^2}{ K^\frac{8s}{3(2s+d)} E^{\frac{s}{2s+d}}}. \label{eq4.27}
		\end{equation}
		
		If $2^{\frac{1}{4}}((8R+5)\pi)^{\frac{1}{3}}K^{\frac{1}{3}} \ge E^{\frac{1}{4}}$, it is straightforward to get \begin{align}
			\|h\|^2_{L^2(D)} &=\int_{|\xi| \le 2 K^\gamma}|\hat{h}(\xi)|^2\,\mathrm{d}\xi+\int_{|\xi| > 2 K^\gamma}|\hat{h}(\xi)|^2\,\mathrm{d}\xi \notag\\ &\lesssim K^{2m+\frac{2d}{d+2s}}\epsilon^2+ R^{2d}\frac{M^2+Q^2}{K^{\frac{4s}{2s+d}}} \lesssim K^{2m+\frac{2d}{d+2s}}\epsilon^2 + (R+1)^{2d-\frac{4s}{3(2s+d)}}\frac{M^2+Q^2}{K^{\frac{8s}{3(2s+d)}}E^\frac{s}{2s+d}}. \label{eq4.28}
		\end{align} Combining \eqref{eq4.27} and \eqref{eq4.28}, we complete the proof.
	\end{proof} 
	
The stability estimate \eqref{stability} consists of the Lipschitz data discrepancy and a logarithmic stability. The latter illustrates the ill-posedness of the inverse source
	problem. Observe that as the upper bound $K$ of the frequency increases, the logarithmic stability decreases which leads to the improvement of the stability estimate.
	Clearly, the stability estimate \eqref{stability} implies the uniqueness of the inverse problem.

\section{Inverse problem using far-field data data}\label{far}
In this section, we consider using the far-field data given by
\begin{equation}\label{eq5.1}
u^\infty(\hat{x}, k) = -C_dk^{\frac{d-3}{2}}\int_{\mathbb R^d}e^{-{\rm i}\hat{x}\cdot y}f(y){\rm d}y.
\end{equation} Here $C_d$ is a constant dependent on the dimension $d$. According to \cite{li2022far}, we have for $\tau>0$ that
\begin{align*}
&\mathbb E[u^\infty(\hat{x}, k+\tau) \overline{u^\infty(\hat{x}, k)}]\\
&=|C_d|^2(k+\tau)^{\frac{d-3}{2}}k^{\frac{d-3}{2}}\Big[\int_{\mathbb R^d}\int_{\mathbb R^d} e^{-{\rm i}(k+\tau)\hat{x}\cdot y}  e^{{\rm i}k\hat{x}\cdot z}
\mathbb E[f(y)f(z)]{\rm d}y{\rm d}z\\
&= |C_d|^2(k+\tau)^{\frac{d-3}{2}}k^{\frac{d-3}{2}}  \int_{\mathbb R^d} \Big[\int_{\mathbb R^d} K_f(y, z)e^{-{\rm i}k\hat{x}\cdot(y-z)}{\rm d}z\Big]
e^{-{\rm i}\tau\hat{x}\cdot y}{\rm d}y \Big)\\
&= |C_d|^2(k+\tau)^{\frac{d-3}{2}}k^{\frac{d-3}{2}} \Big[\int_{\mathbb R^d} h(y) e^{-{\rm i}\tau\hat{x}\cdot y}{\rm d}y |k\hat{x}|^{-m} + \int_{\mathbb R^2} a(y, k\hat{x}) 
e^{-{\rm i}\tau\hat{x}\cdot y}{\rm d}y \Big)\Big]\\
&= |C_d|^2 \Big(\frac{k}{k+\tau}\Big)^{\frac{3-d}{2}}k^{d-3-m}\widehat{h}(\tau\hat{x}) + M{O}(k^{d-4-m}).
\end{align*} Hence we have \begin{align*}
|\widehat{h}(\tau\hat{x})|\lesssim \Big|k^{m+3-d}\mathbb E[u^\infty(\hat{x}, k+\tau) \overline{u^\infty(\hat{x}, k)}]\Big| +\frac{M}{k},
\end{align*}which gives \begin{align*}
|\widehat{h}(\xi)|^2\lesssim \sup_{0<\eta<1, \hat{x}\in\mathbb S^2} \Big|k^{m+3-d}\mathbb E[u^\infty(\hat{x}, (1+\eta)k) \overline{u^\infty(\hat{x}, k)}]\Big|^2
+ \frac{M^2}{k^{2}}, \quad\text{for all} ~ |\xi|\leq k.
\end{align*}
Based on the above estimate,
we introduce the data discrepancy in a finite interval $I = [0, K]$ with $0<K$ as follows
\[
\tilde{\epsilon}^2 =  \sup_{k\in I, \eta\in(0, 1), \hat{x}\in\mathbb S^2} \tilde{\epsilon}^2(k, \eta, \hat{x})
\]
where
\[
\tilde{\epsilon}^2(k, \eta, \hat{x}) := \Big|k^{m+3-d}\mathbb E[u^\infty(\hat{x}, (1+\eta)k) \overline{u^\infty(\hat{x}, k)}]\Big|^2.
\] As the source function is real-valued, we have $\overline{u(x, k)} = u(x,-k)$ and then $\overline{u^\infty(x, k)} = u^\infty(x, -k)$.
Hence, we can analytically extend $\tilde{\epsilon}^2(\cdot,\eta, \hat{x})$ from $\mathbb R^+$ to $\mathbb C$ as follows
\[
\tilde{\epsilon}^2(k, \eta, \hat{x}) = k^{2(m+3-d)}  \mathbb E[u^\infty(\hat{x}, (1+\eta)k) u^\infty(\hat{x}, -k)]
 \mathbb E[u^\infty(\hat{x}, -(1+\eta)k) u^\infty(\hat{x}, k)], \quad k\in\mathbb C,
\]
by noticing that for $k>0$
\[
\tilde{\epsilon}^2(k, \eta, \hat{x})=  \Big|k^{m+3-d}\mathbb E[u^\infty(\hat{x}, (1+\eta)k) \overline{u^\infty(\hat{x}, k)}]\Big|^2.
\] Recall $\mathcal{R}=\{z\in \mathbb{C}:|\arg z|<\pi/4\}$. In order to apply Lemma \ref{le4.3} for $\tilde{\epsilon}^2(k, \eta, \hat{x})$ in $\mathcal{R}$, we shall show $\tilde{\epsilon}^2(k, \eta, \hat{x})$ is analytic and bounded for $k \in \mathcal{R}$.
Recalling \eqref{eq5.1}, for $k \in \mathcal{R}$, we have \begin{align*}&\Big|\mathbb E[u^\infty(\hat{x}, (1+\eta)k) u^\infty(\hat{x}, -k)]\Big| \lesssim |k|^{d-3}\Big|\mathbb E\int_D e^{-ik(1+\eta)\hat{x}\cdot y}f(y)\,\mathrm{d}y\int_D e^{ik\hat{x}\cdot z}f(z)\,\mathrm{d}z\Big| \\ &\lesssim |k|^{d-3}\int_D\int_D|e^{-ik(1+\eta)\hat{x}}e^{ik\hat{x}\cdot z}K(y,z)|\,\mathrm{d}y\,\mathrm{d}z \\ &\lesssim |k|^{d-3}e^{3D_0|\Re k|}\|K(y,z)\|_{L^1(D\times D)}
\end{align*} with  $D_0=diam(D)$.
 Lemma \ref{le4.1} gives that the kernel is weakly singular and \[\|K(y,z)\|_{L^1(D\times D)}\lesssim M.\]
Hence, we have the estimate \begin{equation}\label{eq5.3}|\mathbb E[u^\infty(\hat{x}, (1+\eta)k) u^\infty(\hat{x}, -k)]| \lesssim |k|^{d-3}e^{3D_0|\Re k|} M.
\end{equation} 
The above arguments guarantee the analyticity and boundedness of the data. In summary, we arrive at the following lemma.
\begin{Le}\label{le5.1}
Suppose that the random field $f$ satisfies assumption (A). We have that $\tilde{\epsilon}^2(k, \eta, \hat{x})$ is analytic with respect to $k \in \mathcal{R}$ with the following upper bound \[\tilde{\epsilon}^2(k, \eta, \hat{x})\lesssim |k|^{2m}e^{6D_0|\Re k|} M^2.\]
\end{Le}
Using Lemma \ref{le4.3} and Lemma \ref{le5.1} we have the following lemma.
\begin{Le}
		Let $f$ satisfy assumption (A). Then we have the estimte \[
		|\tilde{\epsilon}^2(k,\theta_1,\theta_2)| \lesssim M^2 \tilde{\epsilon}^{2\mu(k)} e^{(6D_0+1)k},\quad k \in (K,\infty)
		\] with the function $\mu(k)$ satisfying
		\begin{align*}
			\left\{\begin{array}{cc}
				\mu(k) \ge 1/2, & k \in (K, 2^{1/4} K), \\ 
				\mu(k) \ge \pi^{-1}((k/K)^4-1)^{-1/2},  &  k \in (2^{1/4}K, \infty).
			\end{array}\right.
		\end{align*} .\end{Le} Following the arguments in the proof of Theorem \ref{thm4.1} in a straightforward way, we have the following increasing stability estimate by far-field data. The proof is omitted for 
		brevity.
\begin{Thm}
		Let the random function $f$ satisfies assumption (A) and assume $h \in \mathcal{C}_Q$. Then there holds the inequality \begin{align} \label{farstab}
			\|h\|^2_{L^2(D)} \lesssim K^{\frac{2d}{d+2s}}\tilde{\epsilon}^2 + \frac{Q^2+M^2}{K^{\frac{8s}{3(2s+d)}}E^\frac{s}{2s+d}},
		\end{align} where $E=|\log{\tilde{\epsilon}}|$.
	\end{Thm}

	\section{Conclusion}\label{conclusions}
	We establish an increasing stability of an inverse random source problem for the Helmholtz equation. The analysis employs properties of the covariance kernel and the explicit Green function. A possible extension of the current work is to investigate the increasing stability of inverse random source problem in an inhomogeneous media, where the explicit Green function is no longer available. Another interesting topic is the stability for the nonlinear inverse random potential problem. We hope to report the progress of these problems in forthcoming works.

			\bibliographystyle{siam}
			\bibliography{reference}

		
	\end{document}